\documentclass[10pt]{article}
\font\smallit=cmti10

\usepackage{amsmath}
\usepackage{amssymb}
\usepackage{amsfonts}
\usepackage{graphicx}
\usepackage{amsthm}
\usepackage{setspace}
\usepackage{url}
\usepackage{verbatim}
\usepackage[square,numbers,sort&compress]{natbib}

\newtheorem{theorem}{Theorem}
\newtheorem{lemma}{Lemma}[theorem]
\newtheorem{proposition}{Proposition}
\newtheorem{corollary}{Corollary}
\theoremstyle{definition}
\newtheorem{mydef}{Definition}

\makeatletter

\renewcommand\section{\@startsection {section}{1}{\z@}
{-30pt \@plus -1ex \@minus -.2ex}
{2.3ex \@plus.2ex}
{\normalfont\normalsize\bfseries}}

\renewcommand\subsection{\@startsection{subsection}{2}{\z@}
{-3.25ex\@plus -1ex \@minus -.2ex}
{1.5ex \@plus .2ex}
{\normalfont\normalsize\bfseries}}

\renewcommand{\@seccntformat}[1]{\csname the#1\endcsname. }

\makeatother

\begin{document}


\begin{center}
\uppercase{\bf Generalized Nonaveraging Integer Sequences}
\vskip 20pt
{\bf Dennis Tseng}\\
{\smallit Massachusetts Institute of Technology, Cambridge, MA 02139, United States}\\
{\tt dtseng@mit.edu}\\ 
\end{center}
\vskip 30pt

\vskip 30pt 

\centerline{\bf Abstract}

\noindent
Let the sequence $S_{m}$ of nonnegative integers be generated by the following conditions: Set the first term $a_{0}=0$, and for all $k\geq 0$, let $a_{k+1}$ be the least integer greater than $a_{k}$ such that no element of $\{a_{0},\ldots,a_{k+1}\}$ is the average of $m-1$ \emph{distinct} other elements. Szekeres gave a closed-form description of $S_{3}$ in 1936, and Layman provided a similar
description for $S_{4}$ in 1999. We first find closed forms for some similar greedy sequences that avoid averages in terms not \emph{all the same}. Then, we extend the
closed-form description of $S_{m}$ from the known cases when $m=3$
and $m=4$ to any integer $m\geq 3$. With the help of a computer, we also generalize this to sequences that avoid solutions to specific weighted averages in distinct terms. Finally, from the closed forms of these sequences, we find bounds for their growth rates. 

\pagestyle{myheadings}
\thispagestyle{empty}
\baselineskip=12.875pt
\setcounter{page}{1} 
\vskip 30pt 


\section{Introduction} 
Often in combinatorial number theory, we wish to find the maximum number of integers that can be chosen from $\{0,1,\ldots,n-1\}$ without creating a solution to some linear equation in the chosen integers. Ruzsa initiated a systematic study of this problem over all linear equations \citep{Ruzsa93, Ruzsa95}, and the problem has also been extended to systems of linear equations \citep{Shapira06,Koester08}. A couple well-studied examples include constructing sets of integers without three-term arithmetic progressions, which corresponds to avoiding solutions to $x_{1}+x_{2}-2x_{3}=0$, and constructing Sidon sets, which are defined by having no nontrivial solutions to $x_{1}+x_{2}-x_{3}-x_{4}=0$. One way to approach this problem is through the use of a greedy algorithm.

Given an integer $m\geq 3$, define the sequence $S_{m}$
of nonnegative integers by the following conditions:

(i) $a_{0}=0$

(ii) Having chosen $a_{0},a_{1},\ldots,a_{k}$, let $a_{k+1}$ be the
least integer greater than $a_{k}$ such that there are no
\emph{distinct}
$x_{1},x_{2},\ldots,x_{m}\in\{a_{0},a_{1},\ldots,a_{k+1}\}$ with
\begin{equation*}
  x_1+\cdots+x_{m-1}=(m-1)x_{m}.\\
\end{equation*}
The sequence $S_{m}$ constructs a sequence of integers that avoids solutions to $x_{1}+\cdots+x_{m-1}=(m-1)x_{m}$ using a greedy algorithm. Generating $S_{3}$, which avoids three-term arithmetic progressions, we
obtain

$0,1,3,4,9,10,12,13,27,28,30,31,36,37,39,40,81\ldots$

There is an alternative definition for $S_{3}$. An integer is in
$S_{3}$ if and only if there is no 2 in its representation in base
3. This follows from a more general result, as Erd\H{o}s and Tur\'{a}n
\citep{Erdos36} wrote that Szekeres showed the use of the
greedy algorithm to avoid $m$-term arithmetic progressions, for $m$
prime, results in a sequence that contains the integers that do not
contain the digit $m-1$ when expressed in base $m$.

The nice closed-form
description suggests that we can extend this to more general
averages. The sequence $S_{4}$ has a similar closed-form
description as $S_{3}$. The following theorem is due to Layman
\citep{Layman98}.
\begin{theorem}
	\label{Layman}
  An integer is in $S_{4}$ if and only if it can be
  written in the form $M+r$, where the base 4 representation of $M$ has
  only 3's and 0's and ends with a 0, and $r$ is any integer from 0 to 4
  inclusive.
\end{theorem}

Extending this generalization will form the basis of the rest of our
investigation. In Section \ref{section: variant}, we present the
closed forms of some related sequences that avoid solutions to weighted averages in terms \emph{not all the same}.
Then in Section \ref{section: mainresult}, we prove a result that can be used to find
the closed forms of $S_{m}$ for all $m\geq 3$ and the closed forms of
sequences that avoid solutions to specific weighted averages. Finally
in Section \ref{section: asymptotics}, given the closed forms, we can derive bounds 
that allows us to show how efficient the greedy algorithm is asymptotically.
\subsection{Definitions}

We make some definitions to simplify the notation for the rest of the
paper. Unless otherwise stated, for an ordered tuple
$E=(d_{1},\ldots,d_{m-1})$, we will assume throughout the paper that
$1\leq d_{1}\leq d_{2}\ldots\leq d_{m-1}$, i.e. the components are
arranged in nondecreasing order. Let the ordered tuple
$E_{m}=(1,1,\ldots,1)$, where there are $m-1$ components in the tuple.

\begin{mydef} 
  Given an ordered tuple $E=(d_{1},\ldots,d_{m-1})$, let
  $d(E)=d_{1}+\cdots+d_{m-1}$. When the choice of $E$ is obvious, we
  will simply denote $d(E)$ as $d$.
\end{mydef}

\begin{mydef} 
  Call an ordered tuple of positive integers
  $E=(d_{1},\ldots,d_{m-1})$ \emph{valid} if and only if the following
  conditions are satisfied:

  (i) $1=d_{1}$.

  (ii) $d_{2}\leq d_{1},d_{3}\leq d_{1}+d_{2},\ldots,d_{m-1}\leq
  d_{1}+\ldots+d_{m-2}$.\\ In particular, this implies that
  $d_{1}=d_{2}=1$.
\end{mydef}

For example $E_{3}=(1,1)$. Also, $E_{m}$, for all $m\geq 3$, and
$(1,1,2,4,8)$ are valid ordered tuples, while $(1,1,3)$ is not a valid
ordered tuple.

\subsection{Definition of Sequences} 

In this paper, we will focus on finding closed forms for the following
sequences.

\begin{mydef} 
  Given an ordered tuple $E=(d_{1},\ldots,d_{m-1})$, define the
  sequence $A_{E}$ of nonnegative integers by the following
  conditions:

  (i) $a_{0}=0$

  (ii) Having chosen $a_{0},a_{1},\ldots,a_{k}$, let $a_{k+1}$ be the
  least integer greater than $a_{k}$ such that there are no terms
  $x_{1},x_{2},\ldots,x_{m}\in\{a_{0},\ldots,a_{k+1}\}$, \emph{not all
    the same}, that satisfy $d_{1}x_{1}+\cdots+d_{m-1}x_{m-1}=dx_{m}$.
\end{mydef}

\begin{mydef} 
  Given an ordered tuple $E=(d_{1},\ldots,d_{m-1})$, define the
  sequence $S_{E}$ of nonnegative integers by the following
  conditions:

  (i) $a_{0}=0$

  (ii) Having chosen $a_{0},a_{1},\ldots,a_{k}$, let $a_{k+1}$ be the
  least integer greater than $a_{k}$ such that there are no
  \emph{distinct} terms
  $x_{1},x_{2},\ldots,x_{m}\in\{a_{0},\ldots,a_{k+1}\}$ that satisfy
  $d_{1}x_{1}+\cdots+d_{m-1}x_{m-1}=dx_{m}$.\\
\end{mydef}

To simplify notation, we will refer to the sequences $S_{E_{m}}$ and 
$A_{E_{m}}$ for integer $m\geq 3$ as simply $S_{m}$ and $A_{m}$ respectively.

\section{Analysis of the Sequences $A_{E}$}\label{section: variant}

\subsection{A Property of Valid Ordered Tuples} 

We will prove a property of valid ordered tuples that we will use
throughout the paper.

\begin{proposition}\label{prop: valid}
  An ordered tuple $E=(d_{1},\ldots,d_{m-1})$ is valid if and only if
  for every integer $0\leq j\leq d-1$, there exists a subset $H_{j}$
  of $\{2,\ldots,m-1\}$ such that $\displaystyle\sum_{k\in
    H_{j}}{d_{k}}=j$.
\end{proposition}

\begin{proof} 
  We will show that, given a valid ordered tuple
  $E=(d_{1},\ldots,d_{m-1})$, there exists a subset
  $H_{j}\subset\{2,\ldots,m-1\}$ for every integer $0\leq j\leq d-1$
  by induction. For the base case, $H_{0}=\{\}$ and
  $H_{1}=\{2\}$. Now, assume that for some integer $3\leq l\leq m-1$,
  we have found a subset $H_{j}$ of $\{2,\ldots,l-1\}$ for all $0\leq
  j\leq \displaystyle\sum_{k=2}^{l-1}{d_{k}}$. Let $j$ be an integer
  with $1+\displaystyle\sum_{k=2}^{l-1}{d_{k}}\leq j\leq\displaystyle\sum_{k=2}^{l}{d_{k}}$.

  Then, let $H_{j}=H_{j-d_{l}}\cup\{l\}$. Since $d_{l}\leq\displaystyle\sum_{k=1}^{l-1}{d_{k}}\leq j\leq\displaystyle\sum_{k=2}^{l}{d_{k}}$, $0\leq
  j-d_{l}\leq \displaystyle\sum_{k=2}^{l-1}{d_{k}}$ and $H_{j-d_{l}}$
  must exist by induction. Our induction is complete.

  Now to prove the other direction, let $E=(d_{1},\ldots,d_{m-1})$ be
  any ordered tuple of positive integers such that for every $0\leq
  j\leq d-1$, there exists a subset $H_{j}$ of $\{2,\ldots,m-1\}$ such
  that $\displaystyle\sum_{k\in H_{j}}{d_{k}}=j$. In order for $H_{1}$
  to exist, $d_{2}=1$, which means $d_{1}=1$. Now, assume for the sake
  of contradiction that there is some integer $3\leq l\leq m-1$ such
  that $d_{l}>\displaystyle\sum_{k=1}^{l-1}{d_{k}}$. Then, we cannot
  create the subset $H_{d_{l}-1}$, because the subset $H_{d_{l}-1}$
  cannot contain any integers greater than $l-1$, or else
  $\displaystyle\sum_{k\in H_{d_{l}-1}}{d_{k}}>d_{l}-1$. Also, by
  assumption, $\displaystyle\sum_{k=2}^{l-1}{d_{k}}<d_{l}-1$, so the
  subset $H_{d_{l}-1}$ cannot contain only integers less than or equal
  to $l-1$, which is a contradiction.
\end{proof}

\subsection{Closed Form of $A_{E}$}
\begin{theorem}\label{variantTheorem}
  Given a valid ordered tuple $E$, an integer is in $A_{E}$ if and
  only if it contains only 0's and 1's in its base $d+1$
  representation.
\end{theorem}
\begin{proof}
  Let the sequence $B_{E}$ be the nonnegative integers with only 0's
  and 1's in their base $d+1$ representation in increasing
  order. We show that $B_{E}$ is the same as $A_{E}$. Let
  $E=(d_{1},\ldots,d_{m-1})$.
  \begin{lemma}
    \label{variation Lemma 1} It is impossible to choose $m$ integers
    $x_{1},x_{2},\ldots,x_{m}$, not all the same, that are terms of the
    sequence $B_{E}$ such that
    \begin{equation}
      \label{eq: variation lemma 1 condition}
      d_{1}x_{1}+\cdots+d_{m-1}x_{m-1}=dx_{m}.
    \end{equation}
  \end{lemma}
  \begin{proof} Assume for the sake of contradiction that there are
    $x_{1},\ldots,x_{m}$, not all equal, that satisfy equation \eqref{eq:
      variation lemma 1 condition}. Let $t_{0,k},t_{1,k},\ldots$ be the
    digits of $x_{k}$ in base $d+1$,
    i.e. $x_{k}=\displaystyle\sum_{i=0}^{\infty}{t_{i,k}(d+1)^i}$ for all
    $1\leq k\leq m$. From equation \eqref{eq: variation lemma 1
      condition},
    \begin{equation}
      \sum_{k=1}^{m-1}\sum_{i=0}^{\infty}{d_kt_{i,k}(d+1)^i}=d\sum_{i=0}^{\infty}{t_{i,m}(d+1)^i}.\nonumber
    \end{equation} There is no carrying in base $d+1$ when we add
    $\displaystyle\sum_{k=1}^{m-1}{d_{k}x_{k}}$ because $x_{k}$ contains
    only 0's and 1's in its base $d+1$ representation for all $1\leq k\leq
    m$ and $\displaystyle\sum_{k=1}^{m-1}{d_{k}}<d+1$. Therefore, if
    $t_{i,m}=0$, then $t_{i,k}=0$ for all $1\leq k\leq m-1$. If
    $t_{i,m}=1$, then $t_{i,k}=1$ for all $1\leq k\leq m-1$. But then
    $x_{1}=x_{2}=\ldots=x_{m}$ contradicting the condition that
    $x_{1},x_{2},\ldots,x_{m}$ cannot all be the same.
  \end{proof} Now we show it is impossible to insert terms into
  $B_{E}$, which means $B_{E}$ satisfies the ``greedy'' condition of
  $A_{E}$.
  \begin{lemma}
    \label{variation Lemma 2} Given any integer $x_{1}$ that is not in
    $B_{E}$, we can find terms $x_{2},x_{3},\ldots,x_{m}$ of $B_{E}$, each
    less than $x_{1}$ such that $d_{1}x_{1}+d_{2}x_{2}+\cdots+d_{m-1}x_{m-1}=dx_{m}$.
  \end{lemma}
  \begin{proof} Since $E$ is a valid ordered tuple, by Proposition
    \ref{prop: valid}, for every $0\leq j\leq d-1$, there exists a set
    $H_{j}\subset\{2,\ldots,m-1\}$ such that $\displaystyle\sum_{k\in
      H_{j}}{d_{k}}=j$. Let $t_{0,k},t_{1,k},\ldots$ be the digits of
    $x_{k}$ in base $d+1$,
    i.e. $x_{k}=\displaystyle\sum_{i=0}^{\infty}{t_{i,k}(d+1)^i}$ for all
    $1\leq k\leq m$. For every $i\geq 0$, if $t_{i,1}=0$, then let
    $t_{i,k}=0$ for all $2\leq k\leq m-1$. If $t_{i,1}>0$, let $t_{i,k}=1$
    for all $k\in H_{d-t_{i,1}}$ and $t_{i,k}=0$ for all
    $k\notin H_{d-t_{i,1}}$ so that
    $\displaystyle\sum_{k=1}^{m-1}{d_{k}t_{i,k}}=d$.  Then,
    the sum $\displaystyle\sum_{k=1}^{m-1}{d_{k}x_{k}}$ has only 0's and
    $d$'s when written in base $d+1$. When we
    divide the sum $\displaystyle\sum_{k=1}^{m-1}{d_{k}x_{k}}$ by
    $d$, we obtain an integer that has only 0's and 1's
    when written in base $d+1$, which is in $B_{E}$. Note that
    $t_{i,1}$ must be greater than 1 for some $i=i_{0}$ as $x_{1}$ is not
    a term of $B_{E}$. Then, $t_{i_{0},1}>t_{i_{0},k}$ for all $2\leq
    k\leq m$. Since $t_{i,1}\geq t_{i,k}$ for all $2\leq k\leq m$ and
    $i\geq 0$, $x_{1}>x_{2},\ldots,x_{m}$ as desired.
  \end{proof} Since we have proven no $m$ terms in $B_{E}$ satisfy the
  equation
  $d_{1}x_{1}+\cdots+d_{m-1}x_{m-1}=dx_{m}$
  and no terms can be inserted into $B_{E}$ without creating a solution
  to the equation, $B_{E}$ is the same sequence as $A_{E}$.
\end{proof}

\subsection{A Property of the Sequence $A_{E}$} By
Theorem~\ref{variantTheorem}, the term $a_{n}$ of $A_{E}$ can be found
by writing $n$ in binary and reading it in base
$d+1$. Then, the following result quickly follows.

\begin{proposition}
  \label{prop1} The number of 1's in the base 2 representation of $n$
  is congruent modulo $d$ to the $n^{th}$ term of
  $A_{E}$.
\end{proposition}

\begin{proof} Write $n=\displaystyle\sum_{i=0}^{\infty}{t_{i}2^{i}}$, with
  $t_{0},t_{1},\ldots$ as its digits in base 2. Then,
  $a_{n}=\displaystyle\sum_{i=0}^{\infty}{t_{i}(d+1)^i}\equiv
  \displaystyle\sum_{i=0}^{\infty}{t_{i}}\pmod{d}$.
\end{proof}

\begin{corollary} The terms of $A_{3}$ modulo 2 is the \emph{Thue-Morse sequence}, where the
  $n^{th}$ term is a 0 if $n$ has an even number of 1's in its binary
  expansion and a 1 otherwise by Proposition 1 in \citep{Allouche99}.\end{corollary}

\section{Analysis of the Sequences $S_{E}$}
\label{section: mainresult} We first give an alternative way to
represent the nonnegative integers.  \setcounter{lemma}{2}
\begin{proposition}
  \label{Lemma 1 0} Given positive integers $M\geq 2$ and $c$, every
  nonnegative integer $x$ can be expressed in the form
  $x=c\displaystyle\sum_{i=0}^{\infty} {t_{i}M^{i}}+r$ in exactly one
  way, with integer $0\leq r<c$ and sequence $t_{0},t_{1},\ldots$ such
  that $t_{i}\in\{0,\ldots,M-1\}$ for all $i\geq 0$.
\end{proposition}

\begin{proof} Given a positive integer $x$, let $r_{0}$ and $m_{0}$ be
  the remainder and quotient when $x$ is divided by $c$. So
  $x=r_{0}+cm_{0}$ and $r_{0}$ and $m_{0}$ are uniquely defined. Then
  $r=r_{0}$, and the digits of $m_{0}$ in base $M$ is the sequence
  $t_{0},t_{1},\ldots$, which also must be uniquely defined.
\end{proof}

We now present our main result, which can be used to find closed forms
of the sequences $S_{E}$ for specific choices of $E$.

\begin{theorem}
  \label{conditionTheorem} For some positive integer $z$ and some
  sequence $S_{E}$ for valid ordered tuple $E$, let the set $R_{E}$ be
  $\{a_{0},\ldots,a_{z}\}$ and the constant $c_{E}=a_{z+1}$. \\ Let
  $\max(R_{E})$ denote the maximum element $a_{z}$. Suppose the
  following conditions (i) and (ii) are satisfied:

  (i)
  $c_{E}=1+d\max(R_{E})-\displaystyle\sum_{k=2}^{m-1}{d_{k}(m-k-1)}$.

  (ii) For every integer $0\leq r_{1}\leq c_{E}-1$ and every integer
  $0\leq j\leq d-2$, there exists a subset $H_{j}$ of
  $\{2,\ldots,m-1\}$ and terms $r_{2},\ldots,r_{m}\in R_{E}$ such that
  $\displaystyle\displaystyle\sum_{k\in H_{j}}d_{k}=j$,
  $\displaystyle\sum_{k=1}^{m-1}{d_{k}r_{k}}=dr_{m}$,
  all elements of $\{r_{k}:k\in H_{j}\}\cup\{r_{m}\}$ are distinct, and
  all elements of $\{r_{k}:k\notin H_{j}\cup\{1,m\}\}$ are distinct.\\
  Then all terms in the sequence $S_{E}$ can be expressed in the form
  \begin{equation}
    \label{eq: system form}
    c_{E}\displaystyle\sum_{i=0}^{\infty}{t_{i}(d+1)^i}+r,
  \end{equation} such that $t_{i}=0$ or 1 for all $i$ and $r\in
  R_{E}$.
\end{theorem}

We make a few notes before presenting the proof. First, in order to simply notation, we will drop the subscripts on $c_{E}$ and $R_{E}$ when the choice of $E$ is obvious. Also, we will denote $c_{E_{m}}$ and $R_{E_{m}}$ for all integer $m\geq 3$ as simply $c_{m}$ and $R_{m}$.

Next, Theorem \ref{Layman} is a special case of Theorem \ref{conditionTheorem}. As we will show in Section \ref{Form of S_{m}}, if $E=(1,1,1)$, then we can have $c_{4}=12$ and $R_{4}=\{0,1,2,3,4\}$. If $N=\displaystyle\sum_{i=0}^{\infty}{t_{i}4^i}$ is a nonnegative integer with 0's and 1's as digits when expressed in base 4, then $c_{4}N$ has 0's and 3's as digits and ends in a 0 in base 4. As $N$ ranges over all nonnegative integers with 0's and 1's as digits when expressed in base 4 and $r$ ranges over all elements of $R_{4}$, $c_{4}N+r$ ranges over exactly the same values as described by Layman in Theorem \ref{Layman}. 

Also, given $E$, the choice of $c$ and $R$ is not unique. Using the example where $E=(1,1,1)$ above, we could also let $c_{4}=48$ and $R_{4}=\{0,1,2,3,4,12,13,14,15,16\}$, where Theorem $\ref{conditionTheorem}$ would still predict the same terms for the sequence $S_{4}$. Therefore, given $E$, we will use the minimum value of $c$ that satisfies Theorem \ref{conditionTheorem}.

\begin{proof} Let $\mathcal{B}_{E}$ be the sequence of all integers
  that can be expressed in the form
  $c\displaystyle\sum_{i=0}^{\infty} {t_{i}(d+1)^{i}}+r$,
  \emph{with $t_{i}=0$ or $1$ for all $i\geq 0$ and $r\in R$},
  arranged in increasing order. We prove that $\mathcal{B}_{E}$ is the
  same sequence as $S_{E}$.
  \begin{lemma}
    \label{Lemma 1 3} There are not distinct terms
    $x_{1},x_{2},\ldots,x_{m}$ in $\mathcal{B}_{E}$ such that $d_{1}x_{1}+\cdots+d_{m-1}x_{m-1}=dx_{m}$.
  \end{lemma}
  \begin{proof} We prove this by contradiction. Assume there are $m$
    distinct numbers $x_1, x_2,\ldots,x_{m}$ in $\mathcal{B}_{E}$ such
    that
    \begin{equation}
      \label{eq: assumption1}
      d_{1}x_{1}+\cdots+d_{m-1}x_{m-1}=dx_{m}.
    \end{equation} Because $x_1, x_2,\ldots,x_{m}$ are in
    $\mathcal{B}_{E}$, we can express $x_k
    =c\displaystyle\sum_{i=0}^{\infty}
    {t_{i,k}(d+1)^{i}}+r_{k}$, with $t_{i}=0$ or $1$ for all
    $i\geq 0$ and $r\in R$, for all $1\leq k\leq m$. Let
    $X=dx_{m}$ and express $X$ as
    $c\displaystyle\sum_{i=0}^{\infty}
    {T_{i}(d+1)^{i}}+\mathcal{R}$ such that
    $T_{i}=dt_{i,m}$ for all $i\geq 0$ and
    $\mathcal{R}=dr_{m}$. Because of equation \eqref{eq:
      assumption1},
    \begin{equation}
      \label{eq: assumption1again} c\sum_{i=0}^{\infty}{T_{i}(d+1)^{i}}+\mathcal{R}=c\left(\sum_{k=1}^{m-1}{\sum_{i=0}^{\infty}{d_{k}t_{i,k}(d+1)^i}}\right)+\sum_{k=1}^{m-1}{d_{k}r_{k}}.
    \end{equation} 
  
    If $\mathcal{R}\neq\displaystyle\sum_{k=1}^{m-1}{d_{k}r_{k}}$,
    then $\mathcal{R}-\displaystyle\sum_{k=1}^{m-1}{d_{k}r_{k}}$ is a
    multiple of $c$ or equation \eqref{eq: assumption1again} cannot be
    satisfied. Since both $\mathcal{R}$ and
    $\displaystyle\sum_{k=1}^{m-1}{d_{k}r_{k}}$ are bounded above and
    below by $d\max(R)$ and 0, the difference between
    $\mathcal{R}$ and $\displaystyle\sum_{k=1}^{m-1}{d_{k}r_{k}}$ is at
    most $d\max(R)$. 
    
    We show that $d\max(R)<2c$. By condition
    (i), $2c>2d\max(R)-2\displaystyle\sum_{k=2}^{m-1}{d_{k}(m-k-1)}$. Then,
    since $\max(R)\geq m-2$ and
    $\left(\displaystyle\sum_{k=2}^{m-1}{d_{k}}\right)(\frac{0+m-3}{2})\geq
    \displaystyle\sum_{k=2}^{m-1}{d_{k}(m-k-1)}$ by the rearrangement
    inequality,
    \begin{align*}
      2c&>2d\max(R_{E})-2\sum_{k=2}^{m-1}{d_{k}(m-k-1)}\\
      2c&>d\max(R_{E})+d(m-2)-2\left(\sum_{k=2}^{m-1}{d_{k}}\right)\left(\frac{0+m-3}{2}\right)\\
      2c&>d\max(R_{E}).
    \end{align*}

    Since $d\max(R)<2c$, $\mathcal{R}$ and
    $\displaystyle\sum_{k=1}^{m-1}{d_{k}r_{k}}$ can differ only by
    $c$.
    
    Therefore, we have 3 cases to consider.\\
    Case 1: $\displaystyle\mathcal{R}=\sum_{k=1}^{m-1}
    {d_{k}r_{k}}$
    
    If $\displaystyle\mathcal{R}=\sum_{k=1}^{m-1}
    {d_{k}r_{k}}$, then we have $\displaystyle\sum_{i=0}^{\infty}
    {T_{i}(d+1)^{i}}=\displaystyle\sum_{k=1}^{m-1}{\displaystyle\sum_{i=0}^{\infty}{d_{k}t_{i,k}(d+1)^i}}$,
    which means $T_{i}=\displaystyle\sum_{k=1}^{m-1}{d_{k}t_{i,k}}$ for
    all $i$ by the same argument we used in Lemma \ref{variation Lemma
      1}. 
    If $T_{i}=0$, then $t_{i,k}=0$
    for all $1\leq k\leq m-1$. If $T_{i}=d$, then 
    $t_{i,k}=1$ for all $1\leq k\leq m-1$. Then, for
    $x_1,x_2,\ldots,x_{m}$ to be distinct, there must be $m$ distinct
    values $r_{1},r_{2},\ldots,r_{m}$ that satisfy equation
    $\mathcal{R}=dr_{m}=\displaystyle\sum_{k=1}^{m-1}
    {d_{k}r_{k}}$. However, this is impossible because
    $r_{1},r_{2},\ldots,r_{m}$ are terms of $S_{E}$.\\
	  Case 2:
    $\mathcal{R}=\displaystyle\sum_{k=1}^{m-1}{d_{k}r_{k}}+c$

    Let $i_{0}$ be the minimum nonnegative integer such that
    $T_{i_{0}}=0$. Subtract $c$ from $\mathcal{R}$, add $1$ to
    $T_{i_{0}}$ and set $T_{i}=0$ for all $i<i_{0}$ so that
    $\mathcal{R}=\displaystyle\sum_{k=1}^{m-1}{d_{k}r_{k}}$ and the value
    of $X$ is unchanged. This process is similar to the process of
    carrying digits upon addition. Therefore, $T_{i_{0}}=1$ and $T_{i}$ is
    0 or $d$ for all $i\neq i_{0}$. Since
    $\mathcal{R}=\displaystyle\sum_{k=1}^{m-1}{d_{k}r_{k}}$,
    $T_{i}=\displaystyle\sum_{k=1}^{m-1}{d_{k}t_{i,k}}$ for all $i$. For
    all $i\neq i_{0}$, if $T_{i}$ is 0, then $t_{i,k}=0$ for all
    $1\leq k\leq m-1$. If $T_{i}=d$, then $t_{i,k}=1$ for
    all $1\leq k\leq m-1$. Finally, $t_{i_{0},k}=0$ for all $1\leq k\leq
    m-1$ except when $k=k_{0}$ for some $k_{0}$, where $d_{k_{0}}=1$ and
    $t_{i_{0},k_{0}}=1$.

    Since $d_{k_{0}}=1$, without loss of generality, we can let
    $k_{0}=1$. Then, $r_{2},\ldots,r_{m-1}$ must be distinct for
    $x_{2},\ldots,x_{m-1}$ to be distinct. So by the rearrangement
    inequality, the minimum value of
    $\displaystyle\sum_{k=1}^{m-1}{d_{k}r_{k}}$ is $0\cdot
    d_{1}+\displaystyle\sum_{k=2}^{m-1}{d_{k}(m-k-1)}$. Also, since
    $\mathcal{R}\leq d\max(R)$ and we subtracted
    $c$ from $\mathcal{R}$, $\mathcal{R}\leq
    d\max(R)-c$. Since
    $\mathcal{R}=\displaystyle\sum_{k=1}^{m-1}{d_{k}r_{k}}$, that means
    $d\max(R)-c\geq
    \displaystyle\sum_{k=2}^{m-1}{d_{k}(m-k-1)}$. However, this
    contradicts condition (i). \\ 
    Case 3: $\mathcal{R}=
    \displaystyle\sum_{k=1}^{m-1}{d_{k}r_{k}}-c$

    Let $T_{i_{0}}$ be the minimum nonnegative integer such that
    $T_{i_{0}}=d$. Add $c$ to $\mathcal{R}$, subtract
    $1$ from $T_{i_{0}}$ and set $T_{i}=d$ for all
    $i<i_{0}$ so that $R_{E}=\displaystyle\sum_{k=1}^{m-1}{d_{k}r_{k}}$
    and the value of $X$ is unchanged. This process is similar to carrying
    digits upon subtraction. So $T_{i_{0}}=d-1$ and $T_{i}$
    is 0 or $d$ for all $i\neq i_{0}$. Since
    $\mathcal{R}=\displaystyle\sum_{k=1}^{m-1}{d_{k}r_{k}}$,
    $T_{i}=\displaystyle\sum_{k=1}^{m-1}{d_{k}t_{i,k}}$ for all $i$. For
    all $i\neq i_{0}$, if $T_{i}$ is 0, then $t_{i,k}=0$ for all $k$. If
    $T_{i}=d$, then $t_{i,k}=1$ for all $k$. Also
    $t_{i_{0},k}=1$ for all $1\leq k\leq m-1$ except when $k=k_{0}$ for
    some $k_{0}$ where $d_{k_{0}}=1$ and $t_{i_{0},k_{0}}=0$.

    Since $d_{k_{0}}=1$, without loss of generality, we can let
    $k_{0}=1$. Then, $r_{2},\ldots,x_{m-1}$ must be distinct so that
    $x_{2},\ldots,r_{m-1}$ are distinct. So by the rearrangement
    inequality, the value of $\displaystyle\sum_{k=1}^{m-1}{d_{k}r_{k}}$
    is less than or equal to
    $d_{1}\max(R)+\displaystyle\sum_{k=2}^{m-1}{d_{k}(\max(R)-m+1+k)}$. Also,
    since $\mathcal{R}\geq 0$ and we added $c$ to $\mathcal{R}$,
    $\mathcal{R}\geq c$. Since
    $\mathcal{R}=\displaystyle\sum_{k=1}^{m-1}{d_{k}r_{k}}$, that means
    \begin{align*} 
    	c&\leq d_{1}\max(R)+\sum_{k=2}^{m-1}{d_{k}(\max(R)-m+1+k)}\\
      c&\leq d\max(R)-\sum_{k=2}^{m-1}{d_{k}(m-k-1)},
    \end{align*} 
    which contradicts condition (i).
  \end{proof}

  To finish the proof of Theorem \ref{conditionTheorem}, we need to
  show that no additional elements can be inserted into
  $\mathcal{B}_{E}$.

  \begin{lemma}
    \label{Lemma 1 4} Given any value $y_{1}$ not a term of
    $\mathcal{B}_{E}$, there are distinct terms $y_{2},y_{3},\ldots,y_{m}$
    of $\mathcal{B}_{E}$, each less than $y_{1}$, such that there is a
    permutation $x_{1},\ldots,x_{m}$ of $y_{1},\ldots,y_{m}$ such that
    $d_{1}x_{1}+\cdots+d_{m-1}x_{m-1}=dx_{m}$.
  \end{lemma}
  \begin{proof} By Proposition \ref{Lemma 1 0}, we can express $y_{1}$
    in the form
    $y_{1}=c\displaystyle\sum_{i=0}^{\infty}{t_{i,1}(d+1)^{i}}+r_{1}$,
    where $t_{i,1}$ is an integer between 0 and $d$
    inclusive for all $i\geq 0$ and $r_{1}$ is an integer between $0$ and
    $c-1$ inclusive.

    Express $y_{k}$, for all $2\leq k\leq m$, as
    $c\displaystyle\sum_{i=0}^{\infty}
    {t_{i,k}(d+1)^{i}}+r_{k}$, where $t_{i,k}$ is 0 or $1$ and
    $r_{k}\in R$ for all $i$.

    If $t_{i,1}$ is 0 or 1 for all $i\geq 0$, let
    $t_{i,1}=\cdots=t_{i,m}$ for all $i\geq 0$. Then, $r_{1}\notin R$
    or else $y_{1}$ is a term of $\mathcal{B}_{E}$. Therefore, we can find
    distinct $r_{2},\ldots,r_{m}$, all less than $r_{1}$, such that there
    exists a permutation $s_{1},\ldots,s_{m}$ of $r_{1},\ldots,r_{m}$ that
    satisfies
    $d_{1}s_{1}+\cdots+d_{m-1}s_{m-1}=ds_{m}$. Finally,
    we can let
    $y_{k}=r_{k}+c\displaystyle\sum_{i=0}^{\infty}{t_{i,1}(d+1)^i}$
    and
    $x_{k}=s_{k}+c\displaystyle\sum_{i=0}^{\infty}{t_{i,1}(d+1)^i}$
    for all $1\leq k\leq m$.

    Now we consider the case when $t_{i,1}>1$ for some
    $i\geq 0$. Let $x_{k}=y_{k}$ for all $1\leq k\leq m$. For all $i\geq 0$, let
    $t_{i,k}=0$ for all $2\leq k\leq m$ if $t_{i,1}=0$. If $t_{i,1}\geq
    1$, then let $t_{i,k}=1$ for all $k\in
    H_{d-t_{i,1}}\cup\{m\}$ and $t_{i,k}=0$ otherwise,
    where $H_{d-t_{i,1}}$ is a subset of $\{2,\ldots,m-1\}$
    such that $\sum_{k\in H_{d-t_{i,1}}}{d_{k}}=d-t_{i,1}$. Pick any
    $i_{0}$ for which $t_{i_{0},1}>1$.

    Let $j=d-t_{i_{0},1}$. By condition (ii), we can
    find a set $H_{j}$ and terms $r_{2},\ldots,r_{m}\in R$ such that
    $\displaystyle\displaystyle\sum_{k\in H_{j}}{d_{k}}=j$,
    $\displaystyle\sum_{k=1}^{m-1}{d_{k}r_{k}}=dr_{m}$,
    all elements of $\{r_{k}:k\in H_{j}\}\cup\{r_{m}\}$ are distinct, and
    all elements of $\{r_{k}:k\notin H_{j}\cup\{1,m\}\}$ are distinct.

    Then, $x_{p}\neq x_{q}$ if $p\in H_{j}\cup\{m\}$ and $q\notin
    H_{j}\cup\{1,m\}$ because $t_{i_{0},j}\neq t_{i_{0},k}$. Also, all
    elements of $\{x_{k}:k\in H_{j}\}\cup\{x_{m}\}$ are distinct, and all
    elements of $\{x_{k}:k\notin H_{j}\cup\{1,m\}\}$ are
    distinct. Therefore, $x_{2},\ldots,x_{m}$ are distinct.

    Finally, since for all $i$, $t_{i,1}\geq t_{i,k}$ and
    $t_{i_{0},1}>t_{i_{0},k}$ for all $2\leq k\leq m$, $x_{1}>x_{k}$ for
    all $2\leq k\leq m$.

    Then, since
    $\displaystyle\sum_{k=1}^{m-1}{d_{k}t_{i,k}}=dt_{i,m}$
    for all $i$ and
    $\displaystyle\sum_{k=1}^{m-1}{d_{k}r_{k}}=dr_{m}$,
    $d_{1}x_{1}+\cdots+d_{m-1}x_{m-1}=dx_{m}$.
  \end{proof}

  Since no $m$ terms in $\mathcal{B}_{E}$ satisfy the equation
  $d_{1}x_{1}+\cdots+d_{m-1}x_{m-1}=dx_{m}$ and no additional terms can be
  inserted without creating a solution to the equation,
  $\mathcal{B}_{E}$ is the same as $S_{E}$ and our proof of Theorem
  \ref{conditionTheorem} is complete.
\end{proof}

This suggests a connection between the sequences $A_{E}$ and $S_{E}$.

\begin{corollary} Given a valid ordered tuple $E$, let
  $\mathcal{A}_{E}$ be the set of the integers in the sequence
  $A_{E}$. Then, if the sequence $S_{E}$ of terms $a_{0},a_{1},\ldots$
  satisfies conditions (i) and (ii) of Theorem \ref{conditionTheorem}
  for some $z$, the set $\{ca+r:a\in \mathcal{A}_{E}, r\in R\}$
  contains the integers in $S_{E}$, where $c=a_{z+1}$ and
  $R=\{a_{k}: 0\leq k\leq z\}$.
\end{corollary}

\subsection{Closed form for $S_{m}$}
\label{Form of S_{m}}
\begin{mydef} Let $N=\{0,1,\ldots,2n-1\}\cup\{2n+1\}$. For every
  integer $m\geq 3$, Table~\ref{tbl: definitionTable} gives the set of
  integers $R_{m}$ and the integer $c_{m}$.
  \begin{table}[hbtp]
    \begin{center}
      \begin{tabular}{|c|c|c|}\hline $R_{m}$ & $c_{m}$ & $m$
        \\ \hline $\{0\}$ & $1$ & $3$\\ $\{0,1,2,3,5,7,13,26,27,28,29,31\}$
        & $122$ & $5$\\ $\{0,1,2,3,4,5,7,10,33,34,35,36,37,38\}$ & $219$ &
        $7$\\ $\{0,1,\ldots,2n\}$ & $2n^2+3n-2$ & $2n$, $n>1$\\
        $N\cup\{3n+1\}\cup\{c+2n^2+5n: c\in N\}$ & $4n^3+12n^2+5n$ & $2n+1$,
        $n>3$\\\hline
      \end{tabular}
      \caption{Definition of $R_{m}$ and $c_{m}$}\label{tbl:
        definitionTable}
    \end{center}
  \end{table}
\end{mydef}

\begin{theorem}
  \label{theTheorem} An integer is in the sequence $S_{m}$ if and
  only if it can be expressed in the form
  \begin{equation}
    \label{eq: B} c_{m}\sum_{i=0}^{\infty} {t_{i}m^{i}}+r,
  \end{equation} where $t_{i}$ can be either $0$ or $1$ for all $i\geq
  0$ and $r\in R_{m}$.
\end{theorem}

\begin{proof} We need to show that conditions (i) and (ii) of Theorem
  \ref{conditionTheorem} are satisfied.

  \begin{lemma}
    \label{Lemma 1 1 and Lemma 1 2} The set $S_{m}\cap\left[0,c_{m}-1\right]$ is the same as $R_{m}$.
  \end{lemma}
  \begin{proof} In the appendix, we prove the case $m=2n$ in Lemma
    \ref{Lemma 1 1} and the case $m=2n+1$, with integer $n>3$, in Lemma \ref{Lemma
      1 2}. The cases for when $m=3,5,7$ are brute forced with a computer.
  \end{proof}

  Since $S_{m}\cap\left[0,c_{m}-1\right]=R_{m}$, we can easily check that
  condition (i) is satisfied.

  Now, we show that condition (ii) is satisfied. We want to show that
  for every integer $0\leq r_{1}\leq c_{m}-1$ and every integer
  $0\leq j\leq m-3$, there exists a subset $H_{j}$ of $\{2,\ldots,m-1\}$
  and terms $r_{2},\ldots,r_{m}\in R_{m}$ such that
  $\left|H_{j}\right|=j$,
  $\displaystyle\sum_{k=1}^{m-1}{r_{k}}=(m-1)r_{m}$, all elements of
  $\{r_{k}:k\in H_{j}\cup\{m\}\}$ are distinct, and all elements of
  $\{r_{k}:k\notin H_{j}\cup\{1,m\}\}$ are distinct. 
  
  Let $j$ be any integer between 0 and $m-3$ inclusive.
  First, we consider the case when $r_{1}\notin R_{M}$. Since
  $E_{m}$ is a valid ordered tuple, we can find a subset $H_{j}$ of
  $\{2,\ldots,m-1\}$ such that $\left|H_{j}\right|=j$. Also, by the definition of the sequence $S_{m}$, for
  every $r_{1}\notin R_{m}$, we can find distinct
  $r_{2},\ldots,r_{m}<r_{1}$ such that
  $\displaystyle\sum_{k=1}^{m-1}{r_{k}}=(m-1)r_{m}$, so that condition
  (ii) is satisfied. Now we consider the case for when $r_{1}\in
  R_{m}$.
  \begin{lemma}
    \label{Lemma 1 3.5} Given any $r_{1}\in R_{m}$, we can find
    $r_{2},r_{3},\ldots,r_{m}\in R_{m}$ such that $r_{2},\ldots,r_{m-1}$ are distinct
    and $\displaystyle\sum_{k=1}^{m-1}{r_{k}}=(m-1)r_{m}$.
  \end{lemma}
  \begin{proof} The result follows immediately from Lemmas \ref{Lemma
      1 3.5 1} and \ref{Lemma 1 3.5 3} in the Appendix, where we prove the
    cases when $m$ is even and $m$ is odd separately.
  \end{proof}

  Let $r_{1}$ be an element of $R_{m}$.  
  By Lemma
  \ref{Lemma 1 3.5}, let $r_{2},\ldots,r_{m}\in R_{m}$ be
  chosen such that $\displaystyle\sum_{k=1}^{m-1}{r_{k}}=(m-1)r_{m}$ and $r_{2},\ldots,r_{m-1}$ are distinct. If
  there is some value $2\leq k_{0}\leq m-1$ for which $r_{k_{0}}=r_{m}$,
  then let $k_{0}\notin H_{j}$. Otherwise, we can let any $j$ integers
  between $2$ and $m-1$ to be in $H_{j}$.

  Since both conditions (i) and (ii) are satisfied, the proof is
  complete.
\end{proof}

\subsection{Closed forms for particular $S_{E}$}
\begin{table}[hbtp]
  \begin{center}
    \begin{tabular}{|l|l|p{6cm}|}\hline $E$ & Closed Form & $R_{E}$\\ \hline 
    	(1, 1, 1) & $12\sum_{i=0}^{\infty}{t_{i}4^i}+r$ &$ r\in\{0, 1, 2, 3, 4\}$ \\ 
    	(1, 1, 2) & $16\sum_{i=0}^{\infty}{t_{i}5^i}+r$ & $r\in\{0, 1, 2, 3, 4\}$ \\
    	(1, 1, 1, 1) & $122\sum_{i=0}^{\infty}{t_{i}5^i}+r$ & $r\in\{0, 1, 2, 3, 5, 7, 13, 26, 27, 28, 29, 31\}$ \\ 
      (1, 1, 1, 2) & $103\sum_{i=0}^{\infty}{t_{i}6^i}+r$ & $r\in\{0, 1, 2, 3, 4, 14, 18, 19, 20, 21\}$\\ 
      (1, 1, 2, 3) & $81\sum_{i=0}^{\infty}{t_{i}8^i}+r$ & $r\in\{0, 1, 2, 3, 4, 14, 17, 31, 130, 131, 132,$ $133, 134, 144, 147\}$\\ 
      (1, 1, 2, 4) & $29\sum_{i=0}^{\infty}{t_{i}9^i}+r$ & $r\in\{0, 1, 2, 3, 4\}$\\ 
      (1, 1, 1, 1, 1) & $25\sum_{i=0}^{\infty}{t_{i}6^i}+r$ & $r\in\{0, 1, 2, 3, 4, 5, 6\}$ \\
      (1, 1, 1, 1, 2) & $31\sum_{i=0}^{\infty}{t_{i}7^i}+r $ & $r\in\{0, 1, 2, 3, 4, 5, 6\}$ \\ 
      (1, 1, 1, 1, 3) & $30\sum_{i=0}^{\infty}{t_{i}8^i}+r $ & $r\in\{0, 1, 2, 3, 4, 5\}$ \\
      (1, 1, 1, 1, 4) & $51\sum_{i=0}^{\infty}{t_{i}9^i}+r $ & $r\in\{0, 1, 2, 3, 4, 6, 7\}$ \\ 
      (1, 1, 1, 2, 2) & $106\sum_{i=0}^{\infty}{t_{i}8^i}+r $ & $r\in\{0, 1, 2, 3, 4, 14, 15, 16\}$ \\ 
      (1, 1, 1, 2, 3) & $1170\sum_{i=0}^{\infty}{t_{i}9^i}+r $ & $r\in\{0, 1, 2, 3, 4, 14, 17, 31, 130, 131, 132,$ $133, 134, 144, 147\}$\\ 
      (1, 1, 1, 3, 3) & $38\sum_{i=0}^{\infty}{t_{i}10^i}+r$ & $r\in\{0, 1, 2, 3, 4, 5\}$ \\ 
      (1, 1, 1, 3, 4) & $43\sum_{i=0}^{\infty}{t_{i}11^i}+r$ & $r\in\{0, 1, 2, 3, 4, 5\}$ \\
      (1, 1, 1, 3, 5) & $48\sum_{i=0}^{\infty}{t_{i}12^i}+r$ & $r\in\{0, 1, 2, 3, 4, 5\}$ \\
      (1, 1, 1, 3, 6) & $653\sum_{i=0}^{\infty}{t_{i}13^i}+r$ & $r\in\{0, 1, 2, 3, 4, 12, 34,42, 48, 55\}$ \\
      (1, 1, 2, 2, 2) & $32\sum_{i=0}^{\infty}{t_{i}9^i}+r$ & $r\in\{0, 1, 2, 3, 4, 5\}$ \\ 
      (1, 1, 2, 2, 3) & $ 208\sum_{i=0}^{\infty}{t_{i}10^i}+r$ & $r\in\{0, 1, 2, 3, 4, 18, 19, 20, 24\}$\\ 
      (1, 1, 2, 2, 5) & $3622\sum_{i=0}^{\infty}{t_{i}12^i}+r$ & $r\in\{0, 1, 2, 3, 4, 19, 22, 28, 50, 300, 301,$ $302, 303, 304, 319, 322, 330\}$ \\ 
      (1, 1, 2, 2, 6) & $52\sum_{i=0}^{\infty}{t_{i}13^i}+r$ & $r\in\{0, 1, 2, 3, 4, 5\}$ \\
      (1, 1, 2, 3, 3) & $401\sum_{i=0}^{\infty}{t_{i}11^i}+r$ &$r\in\{0, 1, 2, 3, 4, 8, 37, 38, 39, 40, 41\}$ \\ 
      (1, 1, 2, 3, 4) & $420\sum_{i=0}^{\infty}{t_{i}12^i}+r$ &$r\in\{0, 1, 2, 3, 4, 23, 35, 37, 39\}$\\
      (1, 1, 2, 3, 7) & $61\sum_{i=0}^{\infty}{t_{i}15^i}+r $ & $r\in\{0, 1, 2, 3, 4, 5\}$ \\
      (1, 1, 2, 4, 4) & $50\sum_{i=0}^{\infty}{t_{i}13^i}+r$ & $r\in\{0, 1, 2, 3, 4, 5\}$ \\ 
      (1, 1, 2, 4, 7) & $80\sum_{i=0}^{\infty}{t_{i}16^i}+r $ & $r\in\{0, 1, 2, 3, 4, 5, 6\}$
      \\\hline
    \end{tabular}
  \end{center}
  \caption{Closed Forms for $S_{E}$}\label{tbl: closedform1}
\end{table} With a computer program, we tested the valid ordered
tuples $E=(d_{1},\ldots,d_{m-1})$ for when $4\leq m\leq 7$ until the
terms exceeded 80,000 to identify closed forms for $S_{E}$ for 129
choices of $E$. The 25 tuples the computer found when $4\leq m\leq 6$
are given in Table \ref{tbl: closedform1}, where $t_{i}=0$ or $1$ for
$i\geq 0$ for each of the closed forms.  
\section{Asymptotics}
\label{section: asymptotics} Let $g(n)$ be the number of terms of
$A_{E}$ that are less than $n$, for some positive real $n$ and valid ordered tuple $E$. Similarly,
let $h(n)$ be the number of terms of $S_{E}$ that are less than
$n$. We will derive bounds for $g(n)$ and $h(n)$ and growth rates
of $A_{E}$ and $S_{E}$.

For any valid ordered tuple $E$ and nonnegative integer $i_{0}$,
$g((d+1)^{i_{0}})=2^{i_{0}}$ because there are $2^{i_{0}}$
numbers that, when expressed in base $d+1$, have at most
$i_{0}$ digits and only 0's and 1's as digits. Therefore for a
nonnegative integer $n$, we have
\begin{align*} 
	2^{\left\lfloor\log_{d+1}(n)\right\rfloor}&\leq g(n)\leq 2^{\left\lceil\log_{d+1}(n)\right\rceil},\\ 
	\frac{1}{2}\cdot2^{\log_{d+1}(n)}&\leq g(n)\leq 2\cdot2^{\log_{d+1}(n)},\\
  \frac{1}{2}n^{\log_{d+1}(2)}&\leq g(n)\leq 2n^{\log_{d+1}(2)}.
\end{align*} 
From these bounds,
$g(n)=\Theta(n^{\log_{d+1}(2)})$. Also, from these
bounds, we can derive bounds for the growth rate of $A_{E}$. Let the
terms of $A_{E}$ be $a_{0},a_{1},\ldots$. Since $g(a_{n})=n$,
\begin{align*} 
	\frac{1}{2}a_{n}^{\log_{d+1}(2)}&\leq n\leq 2a_{n}^{\log_{d+1}(2)},\\
  2^{\log_{2}(d+1)}n^{\log_{2}(d+1)}&\geq a_{n}\geq 2^{-\log_{2}(d+1)}n^{\log_{2}(d+1)}.
\end{align*} 
Therefore, $a_{n}=\Theta(n^{\log_{2}(d+1)})$.

Now, we bound $h(n)$. Suppose that all terms of $S_{E}$ can be
expressed in the form
\begin{equation}
	r+c\displaystyle\sum_{i=0}^{\infty}{t_{i}(d+1)^{i}},
\end{equation}
where $t_{i}$ is 0 or 1 for all $i\geq 0$, $c$ is a constant, and
$r\in R$ for a set $R$ that contains nonnegative integers that
are all less than $c$. Then, for any positive integer multiple
$k_{0}c$ of $c$, we have $h(k_{0}c)=\left|R\right|g(k_{0})$
because there are $g(k_{0})$ ways to choose the sequence
$t_{0},t_{1},\ldots$ and $\left|R\right|$ ways to choose
$r$. Therefore, for a nonnegative integer $n$, we have
\begin{align*} 
	\left|R\right|g(\left\lfloor\frac{n}{c}\right\rfloor)&\leq h(n)\leq\left|R\right|g(\left\lceil\frac{n}{c}\right\rceil),\\
  \left|R\right|g(\frac{n}{c}-1)&\leq h(n)\leq\left|R\right|g(\frac{n}{c}+1),\\
  \left|R\right|\frac{1}{2}(\frac{n}{c}-1)^{\log_{d+1}(2)}&\leq h(n)\leq\left|R\right|2(\frac{n}{c}+1)^{\log_{(d+1)}(2)},\\
  \frac{1}{2}\left|R\right|c^{-\log_{d+1}(2)}(n-c)^{\log_{d+1}(2)}&\leq h(n)\leq 2\left|R\right|c^{-\log_{d+1}(2)}(n+c)^{\log_{d+1}(2)}.
\end{align*} 
From these bounds, we get $h(n)=\Theta(n^{\log_{d+1}(2)})$. Also, from these
bounds, we can derive a bound for the growth rate of $S_{E}$. Let the
terms of $S_{E}$ be $a_{0},a_{1},\ldots$. Since $h(a_{n})=n$,
\begin{align*}
  \left|R\right|\frac{1}{2}(\frac{a_{n}}{c}-1)^{\log_{d+1}(2)}
  &\leq  n\leq\left|R\right|2(\frac{a_{n}}{c}+1)^{\log_{d+1}(2)},\\
  c(\frac{2}{\left|R\right|})^{\log_{2}(d+1)}n^{\log_{2}(d+1)}+c&\geq a_{n}\geq
  c(\frac{1}{2\left|R\right|})^{\log_{2}(d+1)}n^{\log_{2}(d+1)}-c.
\end{align*} 
Therefore, $a_{n}=\Theta(n^{\log_{2}(d+1)})$.

Given a valid ordered tuple $E=(d_{1},\ldots,d_{m-1})$, let $f(n)$ be 
the maximum cardinality over all subsets of $\{0,\ldots,n-1\}$ that do not contain 
a solution to $d_{1}x_{1}+\cdots+d_{m-1}x_{m-1}=dx_{m}$ in elements not all the same.
Milenkovic,
Kashyap, and Leyba \citep{Milenkovic06} showed that Behrend's
construction \citep{Behrend46} can be modified to show that
$f(n)\geq
\gamma_{1}ne^{-\gamma_{2}\sqrt{\ln(n)}-\frac{1}{2}\ln(\ln(n))}(1+o(1))$
for $n>d^2$, where
$\gamma_{1}=d^2\sqrt{\frac{1}{2}\ln(d)}$,
$\gamma_{2}=2\sqrt{2\ln(d)}$, and $o(1)$ vanishes as
$n\rightarrow\infty$. Since $f(n)$ is asymptotically
greater than $g(n)$, for all valid ordered tuples $E$, and $h(n)$,
for all tuples $E$ for which we have a closed form of $S_{E}$, we have
shown that the greedy algorithm is not optimal in these cases.

However, it should be noted that Behrend's construction, while much
stronger asymptotically, is less efficient for small values of
$n$. For example, if we let $E=E_{4}$ and $n=10^{10}$, the bound
obtained by Milenkovic, Kashyap, and Leyba shows that
$f(10^{10})\geq 3187$. The bounds obtained by the
greedy algorithm show $h(10^{10})\geq\left\lceil
  \left|R\right|\frac{1}{2}(\frac{10^{10}}{c}-1)^{\log_{d+1}(2)}\right\rceil=15360$
and $f(10^{10})\geq
g(10^{10})\geq\left\lceil\frac{1}{2}(10^{10})^{\log_{d+1}(2)}\right\rceil=10133$.

\section{Conclusion} We have found the closed forms of all sequences
$A_{E}$, given any valid ordered tuple $E$. Also, we have found the
closed forms of $S_{E}$ for specific choices of $E$, including $E_{m}$
for all $m\geq 3$. Possible future work include simplifying the condition 
needed to be satisfied in Theorem \ref{conditionTheorem} or extending Theorem
\ref{conditionTheorem} to cover more tuples $E$ for when $S_{E}$ has a
closed form. Also, generating the sequences and plotting them suggests
that, in general, there are sequences that cannot be described in a
similar way to our closed forms. Further research can also be done
include in bounding the rates of growth of these sequences. For
example, given an ordered tuple of positive integers
$E=(d_{1},\ldots,d_{m-1})$, it appears that $S_{E}$ grows at least as
fast asymptotically as $S_{m}$.
\section{Acknowledgments}

I would like to thank the Center for Excellence in Education, the
Research Science Institute, and Akamai for funding me in the summer. I
would also like to thank Nan Li for mentoring me, Professor
Richard Stanley for the project idea, and my tutor Dr.~John Rickert. I
am also grateful for the guidance of Professor Jake Wildstrom,
Professor John Layman, and Dr.~Tanya Khovanova during the research
process, for the help of an anonymous referee, Scott Kominers, Wei Lue and Dr.~Johnothon Sauer
during the publication process, and for the advice of Travis Hance,
 in improving the algorithm of my
computer program.


%

\appendix


\section{Appendix}
We present proofs of the Lemmas that were omitted in the main paper. 

\begin{mydef}
Let the set $S_{m}(k)$ contain the terms of $S_{m}$ that are less than
or equal to $k$. 
\end{mydef}

\subsection{Method}
We present a method that will be used repeatedly in the proofs following
Lemmas. 
 
Given integers $\alpha$ and $z$, we need to determine whether
there exist $x_{2},x_{3},\ldots,x_{m}\in
S_{m}(z)$ such that $\alpha+\displaystyle\sum_{k=2}^{m-1}{x_{k}}=(m-1)x_{m}$.

Let the set $W=\{w_{1},w_{2}\ldots,w_{s}\}$ be the set
$S_{m}(z)\backslash\{x_{2},x_{3},\ldots,x_{m-1}\}$. Then, $\alpha+\displaystyle\sum_{k=2}^{m-1}{x_{k}}=(m-1)x_{m}$ is equivalent to
\begin{align}
\label{uberLemma eq 2}
\alpha+\sum_{k=0}^{|S_{m}(z)|-1}{a_{k}}-\sum_{k=1}^{s}{w_{k}}&=(m-1)x_{m}
\end{align}
\subsection{Proofs}

\begin{lemma}
\label{Lemma 1 1}
If $m=2n$ for any integer $n>1$, the the only terms of $S_{m}$
less than $2n^2+3n-2$ is in $\{0,1,\ldots,2n\}$.
\end{lemma}
\begin{proof}
To prove Lemma \ref{Lemma 1 1}, we prove two claims.\bigskip\\
Claim 1: The first $2n+1$ terms of $S_{2n}$ are the integers from 0 to $2n$ inclusive.

The first $2n-1$ terms of $S_{2n}$ are the integers from 0 to $2n-2$
inclusive because there are not enough distinct terms less than $2n-1$
to satisfy the equation $\displaystyle\sum_{k=1}^{2n-1}{x_{k}}=(2n-1)x_{2n}$.

If we substitute $\alpha=2n-1$, $m=2n$, $z=2n-2$, and $W=\{x_{2n}\}$ into
equation \eqref{uberLemma eq 2}, we obtain $x_{2n}=\frac{2n-1}{2}$, which is not
an integer.

If we substitute $\alpha=2n$, $m=2n$, $z=2n-1$, and $W=\{x_{2n},x\}$ into
equation \eqref{uberLemma eq 2}, we obtain $n(2n+1-2x_{2n})=x$. The value of $x$ is between 0 and $2n$ only if $0\leq 2n+1-2x_{2n}\leq
2$. But since $2n+1-2x_{2n}$ is odd, $x=n=x_{2n}$ which is a
contradiction. 
\bigskip\\
Claim 2:  For every value of $2n+1\leq \alpha\leq 2n^2+3n-3$, there are
distinct $x_{2},x_{3},\ldots,x_{2n}\in S_{2n}(2n)$ such that $\alpha+\displaystyle\sum_{k=2}^{2n-1}x_{k}=(2n-1)x_{2n}$.

We find an explicit construction for all $2n+1\leq \alpha\leq 2n^2+3n-3$. If
$2n+1\leq \alpha\leq 2n^2+n-1$, we let
$\alpha=pn+q$, where $0\leq q\leq n-1$. Plug in $m=2n$, $z=\alpha-1$, $W=\{a,b,x_{2n}\}$, and
$x_{2n}=n+c$ in equation \eqref{uberLemma eq 2}, we obtain $(p+1)n+q=2nc+a+b$.
Let $a=0$ if $p$ is odd and $a=n$ if $p$ is even. Let $b=q$ and
$c=\lfloor (p+1)/2\rfloor$ so that $(p+1)n=2nc+a$ and $q=b$, which
satisfies $(p+1)n+q=2nc+a+b$. 

Now we make sure that $x_{2n}$, $a$ and $b$ are distinct. 

If $p$ is even, then $a=n>q=b$ and $x_{2n}=n=a$ only if $c=0$. But
$p\geq 2$ so $c=\lfloor (p+1)/2\rfloor>0$. 

If $p$ is odd, then $x_{2n}=n+c>q=b$. Also $a=b=0$ only if $q=0$, in which case we need to redefine our values of $x_{2n}$, $a$ and $b$ to ensure their distinctness. If $p$ is odd and $q=0$, let $a=2n$, $b=0$ and $c=\lfloor (p+1)/2\rfloor-1$. Then, $a>x_{2n}>b$.

If $2n^2+n\leq \alpha\leq 2n^2+3n-3$, let $x_{2n}=2n$, $b=2n-1$ and
$a=\alpha-(2n^2+n-1)$. Since $a<b<x_{2n}$, $a$, $b$ and $x_{2n}$ are distinct.

From Claim 1 and Claim 2, we have proven
that the integers from $0$ to $2n$ inclusive are in $S_{2n}$ and that
the integers between $2n+1$ and $2n^2+3n-3$ inclusive are not, finishing
the proof for Lemma~\ref{Lemma 1 1}.
\end{proof}

To help prove Lemma \ref{Lemma 1 2}, we prove
Lemma~\ref{Lemma 1 1 2 1}.
\begin{lemma}
\label{Lemma 1 1 2 1}
Given the $2\leq k\leq 2n-2$ consecutive integers $y_{1}<y_{2}<\cdots<y_{k}$ between
$2n-k-1$ and $2n-2$ and
an integer $p$, we can find a set of $k$
integers that does not contain $p$ and is a subset of $S_{2n+1}(2n+1)$
such that the sum of
its elements equal to the sum of the original $k$ consecutive
integers. 
\end{lemma}
\begin{proof}
If $p$ is not one of the integers $y_{1},\ldots,y_{k}$, we are done. If not, let $y_{i_{0}}$ be the median of $\{y_{1},\ldots,y_{k}\}$. If
$p<y_{i_{0}}$, decrement the $p-y_{1}+1$ smallest integers and increment
the $p-y_{1}+1$ largest integers in $\{y_{1},\ldots,y_{k}\}$.

If $p>y_{i_{0}}$, increment the $y_{k}-p+1$ largest integers and
decrement the $y_{k}-p+1$ smallest integers in $\{y_{1},\ldots,y_{k}\}$.

If $p=y_{i_{0}}$, then $k$ must be odd, which means $k<2n-2$ and $y_{1}\geq 2$. Then, decrement the $i_{0}-1$ smallest integers and increment
the $i_{0}$ largest integers in $\{y_{1},\ldots,y_{k}\}$. Then
decrement the smallest integer $y_{1}$ again so $\{y_{1},\ldots,y_{k}\}\subset S_{2n+1}(2n+1)$. 
\end{proof}
\begin{lemma}
\label{Lemma 1 2}
If $n>3$, then $S_{2n+1}(4n^3+12n^2+5n-1)=N\cup\{3n+1\}\cup\{c+2n^2+5n:c\in N\}$, where $N=\{0,1,\ldots,2n-1\}\cup\{2n+1\}$.
\end{lemma}
\begin{proof}
 We start with $a_{0}=0$ and generate the terms to
show they are the terms listed in Lemma~\ref{Lemma 1 2}.

The integers $0,1,\ldots,2n-1$ must be in the $S_{2n+1}$ because there
are not $2n+1$ distinct terms in the
sequence, which means there cannot be distinct terms
$x_{1},x_{2},\ldots,x_{2n+1}$  that satisfy
\begin{equation}
\label{eq: oddrLemmaCondition}
\sum_{k=1}^{2n}{x_{k}}=2nx_{2n+1}.
\end{equation}
Now we show $2n$ cannot be a term of $S_{2n+1}$. If $2n$
were a term of $S_{2n+1}$, we can find a solution for equation \eqref{eq:
  oddrLemmaCondition} by letting $x_{2n+1}=n$ and $x_{k}=k-1$ if $k\leq
n$ and $x_{k}=k$ if $k\geq n+1$.

We use contradiction to prove that $2n+1$ is the next term. If we let
$\alpha=2n+1$, $m=2n+1$, $z=2n$ and $W=\{x_{2n+1}\}$ in equation \eqref{uberLemma eq 2}, we obtain
$x_{2n+1}=n+1/(2n+1)$, which is not an integer. 

We show $3n+1$ is the next term in $S_{2n+1}$. In \eqref{uberLemma eq
  2}, let $\alpha=2n+x$ where $2\leq
x\leq n$, $m=2n+1$, $z=\alpha-1$ and $W=\{x,x_{2n+1}\}$. Then, we obtain
$x_{2n+1}=n+1$, which is in $S_{2n+1}(\alpha-1)$. 

To prove that $3n+1$ is the next term, we again use contradiction. In
equation \eqref{uberLemma eq 2}, let $\alpha=3n+1$, $m=2n+1$, $z=3n$ and
$W\{x,x_{2n+1}\}$. Then, we obtain $n+1+(n+1-x)/(2n+1)=x_{2n+1}$.

Since $0\leq x<2n+1$, the only way $n+1-x$ can be a multiple of $2n+1$ is if
$x=n+1$. But then $x_{2n+1}=n+1=x$, which is a contradiction.

Now, we show that given any $3n+2\leq \alpha<2n^2+5n-1$, we
can find distinct $x_{2},x_{3},\ldots,x_{2n+1}\in S_{2n+1}(3n+1)$
such that $\alpha+\displaystyle\sum_{k=2}^{2n} {x_{k}}=(2n)x_{2n+1}$.

In equation \eqref{uberLemma eq 2}, let $m=2n+1$, $z=\alpha-1$, and $W=\{x,y, x_{2n+1}\}$. Then
we obtain $n+1+\frac{n+1+\alpha-x-y}{2n+1}=x_{2n+1}$.
For every $3n+2\leq \alpha\leq 2n^2-2n-2$, let $\alpha=(2n+1)A+B$, where $2\leq
A\leq n-2$ and $-n\leq B\leq n$. We present the solutions for
$x_{2n+1}$, $x$ and $y$ given $\alpha$ in Table \ref{tbl: Odd Lemma Revised
  Table 1}.
\begin{table}[hbtp]
\begin{center}
\begin{tabular}{c|c|c|c|c}\hline
$\alpha$ & $x_{2n+1}$ & $x$ & $y$&\\ \hline
$A(2n+1)+B$ & $n+1+A$ & 0 & $n+1+B$&$B\leq n-2$, $B\neq A$\\
$A(2n+1)+B$ & $n+1+A$ & 1 & $n+B$&$B=A$, $B\leq n-2$\\
$A(2n+1)+B$ & $n+1+A$ & $B-n+2$ & $2n-1$&$n-1\leq B\leq n$, $A<n-2$\\
$A(2n+1)+B$ & $n+1+A$ & $B-n+3$ & $2n-2$&$n-1\leq B\leq n$, $A=n-2$\\\hline
\end{tabular}
\end{center}
\caption{If $3n+2\leq \alpha\leq 2n^2-2n-2$}\label{tbl: Odd Lemma Revised Table 1}
\end{table}
For every $2n^2-2n+1\leq \alpha\leq 2n^2+5n-1$, let $\alpha=2n^2+C$, where
$-2n+1\leq C\leq 5n-1$. We present the solutions in Table \ref{tbl: Odd Lemma Revised Table 2}.
\begin{table}[hbtp]
\begin{center}
\begin{tabular}{c|c|c|c|c}\hline
$\alpha$ & $x_{2n+1}$ & $x$ & $y$&\\ \hline
$2n^2+C$ & $2n-1$ & $C+2n+5$ & $2n-2$&$-2n-1\leq C\leq -8$\\
$2n^2+C$ & $2n-1$ & $2n-5$ & $2n+1$& $C=-7$\\
$2n^2+C$ & $2n-1$ & $C+n+2$ & $3n+1$&$-6\leq C\leq n-4$\\
$2n^2+C$ & $2n+1$ & $0$ & $C+1$&$n-3\leq C\leq 2n-2$\\
$2n^2+C$ & $2n+1$ & $C-2n+2$ & $2n-1$&$2n-1\leq C\leq 4n-4$\\
$2n^2+C$ & $2n+1$ & $C-3n$ & $3n+1$&$4n-3\leq C\leq 5n-1$\\\hline
\end{tabular}
\end{center}
\caption{If $2n^2-2n-1\leq \alpha\leq 2n^2+5n-1$}\label{tbl: Odd Lemma Revised Table 2}
\end{table}

We show that $2n^2+5n, 2n^2+5n+1,\ldots,2n^2+7n-1, 2n^2+7n+1$ are
the next terms in $S_{2n+1}$ by contradiction. Let
$\alpha\in\{2n^2+5n+c:0\leq c\leq 2n-1\}\cup\{2n^2+7n+1\}$. Assume that there are terms
$x_{2},x_{3},\ldots,x_{2n+1}$ in the sequence, each less than $\alpha$ such
that
\begin{equation}
\label{eq: Odd R Lemma Assumption 1}
\alpha+\sum_{k=2}^{2n}{x_{k}}=(2n)x_{2n+1}.
\end{equation}
We prove that $x_{2n+1}\geq 2n^2+5n$, also by contradiction. Assume that
$x_{2n+1}<2n^2+5n$. If $\alpha$ is the only integer
among $\alpha$, $x_{2},x_{3},\ldots,x_{2n}$ that is greater than or equal to
$2n^2+5n$, then the minimum value for $x_{2n+1}$ is $x_{2n+1}\geq
\frac{2n^2+5n+\sum_{k=0}^{2n-2}{k}}{2n}=2n+1+\frac{1}{2n}
$, which is greater than $2n+1$. So $x_{2n+1}$ can only be $3n+1$. But
since $x_{2},x_{3},\ldots,x_{2n}$ cannot be $3n+1$, by equation \eqref{eq: Odd
  R Lemma Assumption 1},
$3n+1\leq\frac{(2n^2+7n+1)+(2n+1)+\sum_{k=2}^{2n-1}{k}}{2n}=
2n+4+\frac{1}{2n}$, which is a contradiction because $n>3$. If at least one of the integers
$x_{2},x_{3},\ldots,x_{2n}$ are greater than or equal to $2n^2+5n$, then
by equation \eqref{eq: Odd R Lemma Assumption 1} $x_{2n+1}\geq \frac{(2n^2+5n)+(2n^2+5n+1)+\sum_{k=0}^{2n-3}{k}}{2n}= 3n+2+\frac{n+4}{2n}$,
which cannot occur because there are no terms between $3n+2$ and $2n^2+5n-1$ inclusive. Therefore
$x_{2n+1}\geq 2n^2+5n$. 

Let $\alpha=M+r$ such $M$ is $2n^2+5n$ and $r\in S_{2n+1}(2n+1)$ and
$x_{i}=M_{i}+r_{i}$, where $M_{i}$ is 0 or $2n^2+5n$ and
$r_{i}\in S_{2n+1}(3n+1)$. Also,
$r_{i}$ can be $3n+1$ only if $M_{i}=0$. Then,
\begin{align*}
M+r+\sum_{k=2}^{2n}M_{k}+\sum_{k=2}^{2n}r_{k}&=(2n)M_{2n+1}+(2n)r_{2n+1}.
\end{align*}
Since $x_{2n+1}\geq 2n^2+5n$, $M_{2n+1}=2n^2+5n$. The maximum value of
$r+\displaystyle\sum_{k=2}^{2n}{r_{k}}-2nr_{2n+1}$ is less or equal to than twice the sum of the
$n$ largest elements of $S_{2n+1}(3n+1)$, since the minimum value of $2nr_{2n+1}$ is 0 and no three elements of $\{r,r_{2},\ldots,r_{2n}\}$ can be pairwise equal. Otherwise, two elements of $\{\alpha,x_{2},\ldots,x_{2n}\}$ must be equal. 

So the maximum value of the difference is
$2\left((3n+1)+(2n+1)+\displaystyle\sum_{k=n+2}^{2n-1}{k}\right)-2n\cdot 0=3n^2+5n+2$.

Since $3n^2+5n+2<2(2n^2+5n)$, at most one of elements of $\{M_{k}:2\leq
k\leq 2n\}$ can be 0, or else the difference $r+\displaystyle\sum_{k=2}^{2n}{r_{k}}-2nr_{2n+1}$ is less than $2nM_{2n+1}-M-\displaystyle\sum_{k=2}^{2n}{M_{k}}$.
If $\alpha<2n^2+7n-1$, there are not $2n-1$ distinct integers between $2n^2+5n$ and $\alpha-1$ inclusive,
which means $\alpha$ is in $S_{2n+1}$. If $\alpha=2n^2+7n-1$, and not
all $\{M_{k}:2\leq k\leq 2n\}$ are equal to $2n^2+5n$, then the
maximum value for $x_{2n+1}$ would be $x_{2n+1}\leq
\frac{\sum_{k=1}^{2n-1}{2n^2+5n+k}+(3n+1)}{2n}=
2n^2+5n-\frac{3n-1}{2n}$, which is less than $2n^2+5n$, contradicting the assumption that $x_{2n+1}\geq 2n^2+5n$. 

The integer $2n^2+7n$ is not in $S_{2n+1}$ because equation \eqref{eq:
  oddrLemmaCondition} is satisfied if we let $x_{2n+1}=2n^2+6n$ and
$x_{k}=2n^2+5n+k-1$ if $k\leq n$ and $x_{k}=2n^2+5n+k$ if $k\geq n+1$.

If $\alpha=2n^2+7n+1$ and not all elements of $\{M_{k}:2\leq k\leq 2n\}$ are
$2n^2+5n$, then the maximum value for $x_{2n+1}$ is $x_{2n+1}\leq \frac{(2n^2+7n+1)+\left(\sum_{k=2}^{2n-1}{2n^2+5n+k}\right)+(3n+1)}{2n}=2n^2+5n-\frac{n-1}{2n}$,
which is less than $2n^2+5n$, contradicting $x_{2n+1}\geq 2n^2+5n$. If
all $M_{2},M_{3},\ldots,M_{2n}\in\{2n^2+5n\}$, then by assumption \eqref{eq: Odd R Lemma Assumption 1},
\begin{align}
M+r+\sum_{k=2}^{2n}{M_{k}}+\sum_{k=2}^{2n}{r_{k}}&=(2n)M_{2n+1}+(2n)r_{2n+1}\nonumber\\
\label{eq: equation in r1}
r+\sum_{k=2}^{2n}{r_{k}}&=(2n)r_{2n+1}.
\end{align}
But $r$, $r_{2},r_{3},\ldots,r_{2n+1}$ are distinct elements of
$S_{2n+1}(2n+1)$, so equation \eqref{eq: equation in r1} has no solutions and $2n^2+7n+1$ is
in the sequence.

We now show that $\{c: 2n^2+7n+2\leq c\leq 4n^3+12n^2+5n-1\}\cap S_{2n+1}(4n^3+12n^2+5n-1)=\emptyset$. So given any $2n^2+7n+2\leq \alpha\leq
4n^3+12n^2+5n-1$, we show that there are distinct
$x_{2},x_{3},\ldots,x_{2n+1}$ in the sequence such that $\alpha+\displaystyle\sum_{k=2}^{2n}{x_{k}}=(2n)x_{2n+1}$.

For ease of notation, we represent the integers $x_{2},x_{3},\ldots,x_{2n}$ with the two sets
$U=\{u_{1},u_{2},\ldots,u_{p}\}$ and
$V=\{v_{1},v_{2},\ldots,v_{q}\}$. The set $U$ contains the elements of $\{x_{2},x_{3},\ldots,x_{2n}\}$ that are greater than or equal to
$2n^2+5n$, with $2n^2+5n$ subtracted from each those integers. The set
$V$ contains the elements of $\{x_{2},x_{3},\ldots,x_{2n}\}$ that are
less than $2n^2+5n$. All elements in set $U$ must be in $S_{2n+1}(2n+1)$ and all elements in set $V$ must
be in $S_{2n+1}(3n+1)$. 
We can express 
\begin{equation}
\label{eq: Odd Lemma r sum}
\alpha+\sum_{k=2}^{2n}{x_{k}}=\alpha+\sum_{k=1}^{p}{u_{k}}+\sum_{k=1}^{q}{v_{k}}+|U|(2n^2+5n),
\end{equation}
which implies that
\begin{equation}
x_{2n+1}=\frac{\alpha+\sum_{k=1}^{p}{u_{k}}+\sum_{k=1}^{q}{v_{k}}+|U|(2n^2+5n)}{2n}.\nonumber
\end{equation}

The solutions for $\{x_{k}:2\leq k\leq 2n+1\}$ for all $2n^2+7n+2\leq \alpha\leq
2n^2+11n-1$ are displayed in Table \ref{tbl: Odd Lemma Table 9}.
\begin{table}[hbtp]
\begin{center}
\begin{tabular}{c|c|c|c|c}\hline
$\alpha$ & $U$ & $V$ & $x_{2n+1}$\\ \hline
$2n^2+7n+2$ & $S_{2n+1}(2n+1)\backslash\{0,1,3\}$ & $\{2n+1\}$ & $2n^2+5n$&\\
$2n^2+7n+3$ & $S_{2n+1}(2n+1)\backslash\{0,1,2\}$ & $\{2n-1\}$ & $2n^2+5n$&\\
$2n^2+7n+4$ & $S_{2n+1}(2n+1)\backslash\{0,1,3\}$ & $\{2n-1\}$ &
$2n^2+5n$&\\
$2n^2+C$ &\begin{tabular}{c} $S_{2n+1}(2n+1)\backslash$\\$\{0,2,C-7n-2\}$\end{tabular} & $\{2n-1\}$ &
$2n^2+5n$&\begin{tabular}{c}$7n+5\leq C$\\$\leq 9n+1 $\end{tabular}\\
$2n^2+C$ & \begin{tabular}{c}$S_{2n+1}(2n+1)\backslash$\\$\{0,C-9n-1,2n+1\}$\end{tabular} & $\{2n-1\}$ &
$2n^2+5n$&\begin{tabular}{c}$9n+2\leq C$\\$\leq 11n-1$\end{tabular}\\\hline
\end{tabular}
\end{center}
\caption{If $2n^2+7n+2\leq \alpha\leq 2n^2+11n-1$}\label{tbl: Odd Lemma Table 9}
\end{table}
By Lemma~\ref{Lemma 1 1 2 1}, we can define $G(k,p)$ as a subset of $k$
elements of $S_{2n+1}(2n+1)$ that has the same sum as the consecutive integers
between $2n-k-1$ and $2n-2$ inclusive and does not contain the integer
$p$. Since Lemma~\ref{Lemma 1 1 2 1} only applies to when $2\leq k\leq 2n-2$, we need to define $G(k,p)$ for when $k=0$ or $1$. Let $G(0,p)=\{\}$, $G(1,p)=2n-2$ for all $p\neq 2n-2$, and $G(1,2n-2)$ be undefined. 

Also, if $T=\{t_{i}:0\leq i\leq j\}$ is a set of distinct nonnegative integers
arranged in increasing order, we define $H(T)$ to take the smallest
value of $t_{i}>i$ and decrement it. Let $H^{(k)}(T)$ denote applying
the function $H$ to $T$ $k$ times and $\left[n\right]$ be the set containing the
integers from 1 to $n$ inclusive.

The solutions for $\{x_{k}:2\leq k\leq 2n+1\}$ for all $2n^2+11n\leq
\alpha\leq 4n^3+12n^2+n-1$ are displayed in Table \ref{tbl: Odd Lemma Revised
  Table 3}. There may be multiple ways to express $\alpha$ as $2n^2+11n+(2n^2+5n)A+2nB+C$, in which case there are multiple solutions shown. Notice that we cannot have $A=2n-3$ and $B=2n-2$ at the same time, as $G(1,2n-2)$ is undefined. To correct this, we let $A=2n-2$, $B=n-4$, and let $V=H^{(C+n)}(\{2,\ldots,2n-1\}\cup\{2n+1\})$. 
\begin{table}[hbtp]
\begin{center}
\begin{tabular}{c|c|c|c|c}\hline
$\alpha$ & $U$ & $V$ & $x_{2n+1}$\\ \hline
\begin{tabular}{c}$2n^2+11n$\\$+(2n^2+5n)A$\\$+2nB+C$\end{tabular}
& \begin{tabular} {c} $G(2n-2-A$\\$,B)$\end{tabular} &
\begin{tabular}{c}$H^{(C)}(\left[A\right]\cup$\\$\{2n-1\})$\end{tabular} & \begin{tabular}{c}$2n^2+5n$\\$+B$\end{tabular}& \begin{tabular}{c}$0\leq B\leq 2n-1$,\\ $0\leq
  C\leq 2n-1$,\\$0\leq A\leq 2n-2$\end{tabular}\\\hline
\end{tabular}
\end{center}
\caption{If $2n^2+11n\leq \alpha\leq 4n^3+12n^2+n-1$ and $A=2n-3$ and $B=2n-2$ are not true at the same time}\label{tbl: Odd Lemma
  Revised Table 3}
\end{table}

We now present Table \ref{tbl: Odd Lemma Table 16} giving a solution for
every $4n^3+12n^2+n\leq \alpha\leq 4n^3+12n^2+5n-1$.
\begin{table}[hbtp]
\begin{center}
\begin{tabular}{c|c|c|c|c}\hline
$\alpha$ & $U$ & $V$ & $x_{2n+1}$\\ \hline
$4n^3+12n^2+n+C$ & $\emptyset$  & $S_{2n+1}(2n+1)\backslash\{0,C+1\}$ &
$2n^2+7n+1$&$0\leq C\leq2n-2$\\
$4n^3+12n^2+n+C$ & $\emptyset$  & $S_{2n+1}(2n+1)\backslash\{1,2n-1\}$ &
$2n^2+7n+1$&$C=2n-1$\\
$4n^3+12n^2+n+C$ & $\emptyset$  & $S_{2n+1}(2n+1)\backslash\{C-2n,2n+1\}$ &
$2n^2+7n+1$&$2n\leq C\leq 4n-1$\\\hline
\end{tabular}
\end{center}
\caption{If $4n^3+12n^2+n\leq \alpha\leq 4n^3+12n^2+5n-1$}\label{tbl: Odd
  Lemma Table 16}
\end{table}

Since we have worked from 0 to $4n^3+12n^2+5n-1$ and tested if each
integer in that range is in $S_{2n+1}$ and found that the results match
the statement in Lemma~\ref{Lemma 1 2}, our proof is complete. 
\end{proof}

To prove Lemma \ref{Lemma 1 3.5}, we need to prove Lemma~\ref{Lemma 1 3.5 1}
and Lemma~\ref{Lemma 1 3.5 3}. 

\begin{lemma}
\label{Lemma 1 3.5 1}
If $m=2n$, given any $\alpha\in\{0,1,\ldots,2n\}$, we can find distinct
$x_{2},x_{3},\ldots,x_{2n}\in S_{2n}(2n)$ such that $\alpha+\displaystyle\sum_{k=2}^{2n-1}{x_{k}}=(2n-1)x_{2n}$.
\end{lemma}

\begin{proof}
In equation \eqref{uberLemma eq 2}, let $m=2n$, $z=2n$, and $W=\{a,b\}$. Then, we
obtain $\alpha=n(2x_{2n}-2n-1)+a+b$.

We display the solutions for $0\leq \alpha\leq 2n$ in Table \ref{tbl: Revised Lemma 1
  3.5 1 Table 1}.
\begin{table}[hbtp]
\begin{center}
\begin{tabular}{c|c|c|c|c}\hline
$a$ & $b$ & $x_{2n}$\\ \hline
$n$ & $2n$& $n-1$& $\alpha=0$\\
$0$ & $n+\alpha$ & $n$ & $1\leq \alpha\leq n$\\
$0$ & $\alpha-n$ & $n+1$ & $n+1\leq \alpha\leq 2n$\\\hline
\end{tabular}
\end{center}
\caption{If $0\leq \alpha\leq 2n$}\label{tbl: Revised Lemma 1
  3.5 1 Table 1}
\end{table}

Since we have covered all the values for $\alpha$ from 0 to $2n$ inclusive,
we are done with the proof of Lemma \ref{Lemma 1 3.5 1}. 
\end{proof}

To prove Lemma \ref{Lemma 1 3.5 3}, we use of the following result.

\begin{lemma}
\label{Lemma 1 3.5 2}
Given any $\alpha\in S_{2n+1}(2n+1)$, we can find
$x_{2},x_{3},\ldots,x_{2n+1}\in S_{2n+1}(2n+1)$ such that $x_{2},x_{3},\ldots,x_{2n}$
are distinct and $\alpha+\displaystyle\sum_{k=2}^{2n}{x_{k}}=2nx_{2n+1}$.
\end{lemma}

\begin{proof}
In equation \eqref{uberLemma eq 2}, let $m=2n+1$, $z=2n+1$ and $W=\{a,b\}$. Then,
we obtain $\alpha=2nx_{2n+1}+a+b-2n^2-n-1$.

Notice that $x_{2n+1}$ does not necessarily have to be distinct from $a$
and $b$. We display the solutions for $0\leq \alpha\leq n-2$ in Table
\ref{tbl: Revised Lemma 1
  3.5 2 Table 1}.
\begin{table}[hbtp]
\begin{center}
\begin{tabular}{c|c|c|c|c}\hline
$a$ & $b$ & $x_{2n+1}$&\\ \hline
$0$   & $n+1+\alpha$ & $n$&$0\leq \alpha\leq n-2$\\
$1$   & $2n-1$ & $n$& $\alpha=n-1$\\
$\alpha-n$   & $2n+1$ & $n$&$n\leq \alpha\leq 2n-1$\\
$n+1$  & $2n+1$ & $n$& $\alpha=2n+1$\\\hline
\end{tabular}
\end{center}
\caption{If $\alpha\in S_{2n+1}(2n+1)$}\label{tbl: Revised Lemma 1
  3.5 2 Table 1}
\end{table}
 Since we have covered all the cases when
$r\in\{0,1,\ldots,2n-1\}\cup\{2n+1\}$, the proof for Lemma \ref{Lemma 1
    3.5 2} is complete. 
\end{proof}
Now we prove Lemma \ref{Lemma 1 3.5} for the case when $m$ is odd. 
\begin{lemma}
\label{Lemma 1 3.5 3}
Given any $\alpha\in R_{2n+1}$, we can find $x_{2},x_{3},\ldots,x_{2n+1}\in
R_{2n+1}$ such that $x_{2},x_{3},\ldots,x_{2n}$ are distinct and $\alpha+\displaystyle\sum_{k=2}^{2n}{x_{k}}=2nx_{2n+1}$.
\end{lemma}

\begin{proof}
First we prove this for when $n>3$ and then deal with the special cases
when $n\leq 3$. 

If $n>3$ and $\alpha\in S_{2n+1}(2n+1)$, by Lemma \ref{Lemma
  1 3.5 2}, we can select $x_{2},x_{3},\ldots,x_{2n+1}\in S_{2n+1}(2n+1)$ to satisfy the lemma. 

Similarly, if
$\alpha\in\{2n^2+5n+c: 0\leq c\leq 2n-1\}\cup\{2n^2+7n+1\}$, we see
this is the same set as $S_{2n+1}(2n+1)$ with $2n^2+5n$
added to each element. Therefore, also by Lemma \ref{Lemma 1 3.5 2},we can select $x_{2},x_{3},\ldots,x_{2n+1}$ from the set
$\alpha\in\{2n^2+5n+c: 0\leq c\leq 2n-1\}\cup\{2n^2+7n+1\}$ to satisfy the lemma. 

If $\alpha=3n+1$, then let $x_{2}=0$, $x_{k}=k-1$ for all $3\leq k\leq
2n$ and $x_{2n+1}=n+1$. Then, $\alpha+\displaystyle\sum_{k=3}^{2n}{x_{k}}=2nx_{2n+1}$.

If $n=1$, set $x_{2}=x_{3}=0$.

If $n=2$, then the cases for when $\alpha\in\{0,1,2,3,5\}$ and
  $\alpha\in\{26,27,28,29,31\}$ are covered in Lemma \ref{Lemma 1 3.5 2}. If
  $\alpha=3n+1=7$, $7+0+2+3=4\cdot3$. If $\alpha=13$, $13+3+5+7=4\cdot7$. 

If $n=3$, then the only difference between $R_{7}$ and the general definition
for $R_{2n+1}$ when $n>3$ is the the missing 40, so we only need to
consider if $\alpha\in\{33,34,35,36,37,38\}$. The case for when $\alpha\in\{0,1,2,3,4,5,7\}$ is covered in
Lemma \ref{Lemma 1 3.5 2} and if $\alpha=10$, $10+0+2+3+4+5=6\cdot4$. 

The solutions for when $33\leq \alpha\leq 38$ are presented in
table \ref{tbl: Lemma 1 3.5 3 Table 1} below. 
\begin{table}[hbtp]
\begin{center}
\begin{tabular}{c|c|c}\hline
$\alpha$ & $\{x_{2},x_{3},x_{4},x_{5},x_{6}\}$ & $x_{7}$\\ \hline
33 &  $\{10,7,5,4,1\}$  & $10$\\
34 &  $\{10,7,5,4,0\}$  & $10$\\
35 &  $\{10,7,5,3,0\}$  & $10$\\
36 & $\{10,7,5,2,0\}$ & $10$\\
37 & $\{10,7,5,1,0\}$ & $10$\\
38 &  $\{10,7,4,1,0\}$  & $10$\\\hline
\end{tabular}
\end{center}
\caption{If $33\leq \alpha\leq 38$ and $n=3$}\label{tbl: Lemma 1
  3.5 3 Table 1}
\end{table}

Since we have covered all the cases when $m=2n+1$ is odd, the proof for
Lemma \ref{Lemma 1 3.5 3} is complete. 
\end{proof}

\bibliography{biblioIntegers2}
\bibliographystyle{siam}

\end{document}